\documentclass[a4paper,11pt]{amsart}
\usepackage{amsmath,amsthm,amsfonts,amssymb,mathrsfs}
\usepackage{enumerate}
\usepackage[colorlinks=true, linkcolor=blue, citecolor=magenta, menucolor=black]{hyperref}
\usepackage{geometry}
\usepackage[demo]{graphicx}
\usepackage[up]{caption}
\usepackage{subcaption}
\usepackage{pgf,tikz}
\usepackage[toc,page]{appendix}
\usetikzlibrary{arrows}
\usetikzlibrary{patterns}
\usepackage{color}
\usepackage{tikz-cd}

\numberwithin{equation}{section}
\theoremstyle{plain}
\newtheorem{thm}{Theorem}[section]

\newtheorem{lem}[thm]{Lemma}
\newtheorem{prop}[thm]{Proposition}
\newtheorem{cor}[thm]{Corollary}

\theoremstyle{definition}
\newtheorem{defn}[thm]{Definition}
\newtheorem*{defn*}{Definition}

\newtheorem{rem}[thm]{Remark}

\newtheorem{fact}[thm]{Fact}

\newtheorem{claim}[thm]{Claim}

\newcommand{\R}{\mathbb{R}}

\newcommand{\cA}{\mathcal{A}}

\newcommand{\cS}{\mathcal{S}}

\newcommand{\nsrk}{{\rm r}_{ns}}

\allowdisplaybreaks

\begin{document}

\title[]{Non-isomorphism of $A^{*n}, 2\leq n \leq \infty$, for a  non-separable \\  abelian von Neumann algebra $A$}

\author{R\'emi Boutonnet}
\address{Institut de Math\'ematiques de Bordeaux \\ CNRS \\ Universit\'e Bordeaux I \\ 33405 Talence \\ FRANCE}
\email{remi.boutonnet@math.u-bordeaux.fr}
\thanks{RB is supported by ANR grant AODynG 19-CE40-0008}

\author{Daniel Drimbe}
\address{Department of Mathematics, KU Leuven, Celestijnenlaan 200b, B-3001 Leuven, Belgium}
\email{daniel.drimbe@kuleuven.be}
\thanks {DD holds the postdoctoral fellowship fundamental research 12T5221N of the Research Foundation Flanders.}

\author{Adrian Ioana}
\address{Department of Mathematics, University of California San Diego, 9500 Gilman Drive, La Jolla, CA 92093, USA}
\email{aioana@ucsd.edu}
\thanks{AI is supported by NSF DMS grants 1854074 and 2153805, and a Simons Fellowship}

\author{Sorin Popa}
\address{Math Dept, UCLA, 520 Portola Plaza, 
Box 951555, Los Angeles, CA 90095, USA}
\email{popa@math.ucla.edu}
\thanks{SP Supported in part by NSF Grant DMS-1955812  and the Takesaki Endowed Chair at UCLA}

\maketitle

\begin{abstract}
We prove that if $A$ is a non-separable abelian  tracial von Neuman algebra then 
its free powers $A^{*n}, 2\leq n \leq \infty$, 
are mutually non-isomorphic and with trivial fundamental group, $\mathcal F(A^{*n})=1$, whenever $2\leq n<\infty$. This settles the non-separable version of the free group factor  problem. 
\end{abstract}

\section{Introduction}

The free group factor problem, asking whether the II$_1$ factors $L\mathbb F_n$ arising from the free groups with $n$ generators $\mathbb F_n$,  $2\leq n \leq \infty$, are isomorphic or not, 
is perhaps the most famous in operator algebras, being in a way emblematic for this area, broadly known even outside of it.

It is generally believed that the free group factors are not isomorphic.
Since $L\mathbb F_n=L\mathbb Z*\cdots*L\mathbb Z$, this amounts to $A^{*n}, 2\leq n \leq \infty$, 
being non-isomorphic, where $A=L\mathbb Z$ is the unique (up to isomorphism) separable diffuse abelian von Neumann algebra. 
Due to work in \cite{Ra94,Dy94}, based on Voiculescu’s free probability methods, this is also equivalent to 
the fundamental group of $A^{*n}$ being trivial for some (equivalently, all) $2\leq n <\infty$, $\mathcal F(A^{*n})=1$. 

We study here the non-separable version of the free group factor problem, asking whether 
the II$_1$ factors $A^{*n}, 2\leq n\leq\infty$, are non-isomorphic when $A$ is an abelian but non-separable  
von Neumann algebra (always assumed tracial, i.e., endowed 
with a given normal faithful trace).  
Examples of such algebras $A$ include the ultrapower von Neumann algebra 
$(L\mathbb Z)^\omega$ and the group von Neumann algebra $LH$, where $\omega$ is a free ultrafilter on $\mathbb N$ and $H$ is an uncountable discrete abelian group, such as  
$\mathbb R$ or $\mathbb Z^\omega$. 
We obtain the following affirmative answer to the problem: 

\begin{thm}\label{main} Let $A$ be a diffuse non-separable abelian tracial von Neumann algebra. 

Then the II$_1$ factors $A^{*n}, 2\leq n \leq \infty$, are mutually non-isomorphic, and have  
trivial fundamental group, $\mathcal F(A^{*n})=1$, whenever  $2\leq n<\infty$. 
\end{thm}

In other words, if the abelian components of a free product $A^{*n}$ are being ``magnified’’ from separable to non-separable, then 
the corresponding II$_1$ factors do indeed remember the number of terms involved. 
One should note that if $2\leq n\leq\infty$, then any II$_1$ factor $A^{*n}$, with $A$ diffuse abelian,  
is an inductive limit of subfactors isomorphic to $L\mathbb F_n$.

To prove Theorem \ref{main}, we show that the II$_1$ factors of the form  $M=A_1 *\cdots *A_n$, 
with $A_1, A_2,\cdots, A_n$ non-separable abelian, have a remarkably rigid structure. 
Specifically, we prove that given any unital abelian von Neumann subalgebra $B\subset pMp$ that is purely non-separable (i.e., has no separable direct summand) and singular (i.e., has trivial normalizer), there is a partition of $p$ into projections $p_i\in B$ such that $Bp_i$ is unitarily 
conjugate to a direct summand of $A_i$, for every $1\leq i\leq n$ (see Corollary \ref{corner_identification}). 
This implies that the family 
$\{A_ip_i\}_i$, consisting of the maximal purely non-separable direct summands of $A_i$, $1\leq i\leq n$, 
coincides with the {\it sans-core} of $M$, 
a term we use to denote the maximal family $\mathcal A^{ns}_M = \{B_j\}_j$ of pairwise disjoint, singular,   
purely non-separable abelian subalgebras $B_j$ of $M$. The uniqueness (up to unitary conjugacy, cutting and gluing) of this family 
ensures that the {\it sans-rank} of $M$, defined by  $$\text{\rm r}_{ns}(M):=\sum_j \tau(1_{B_j})\in [0,+\infty],$$ 
is an isomorphism  invariant for $M$. This shows in particular that  
if $A$ is a diffuse non-separable abelian von Neumann algebra and $Ap$ is its maximal purely non-separable direct summand,  then  $\text{r}_{\text{ns}}(A^{*n})=n \tau(p)$, for every $2\leq n\leq \infty$, 
implying the non-isomorphism in the first part of Theorem \ref{main}. 
Since the sans-rank is easily seen to satisfy the amplification formula 
$\text{\rm r}_{ns}(M^t)=\text{\rm r}_{ns}(M)/t$, for every $t>0$, the last part of the theorem follows as well. 

We define the sans-core and sans-rank of a II$_1$ factor in Section \ref{sans}, where we also discuss some basic properties, 
including the amplification formula for the sans-rank. 
In Section \ref{mainresults} we prove that $\text{\rm r}_{ns}(*_{i\in I} M_i)=\sum_{i\in I}\text{\rm r}_{ns}(M_i)$, for any family $M_i,i\in I$, of tracial von Neumann algebras (see Theorem \ref{fp sum}) and use this formula to deduce Theorem \ref{main}. The proof of Theorem \ref{fp sum} uses intertwining by bimodules techniques and control of relative commutants 
 in amalgamated free product II$_1$ factors from \cite{IPP05}. 
 Notably, we use results from \cite{IPP05} to show that any von Neumann subalgebra $P$ of a tracial free product $M=M_1*M_2$ which has a non-separable relative commutant, $P'\cap M$,  must have a corner which embeds into $M_1$ or $M_2$ (see Theorem \ref{sep}). The last section of the paper, Section \ref{Remarks and open problems}, records some further remarks and open problems.
 
\section{The singular abelian core of a II$_1$ factor}\label{sans}
The aim of this section is to define the singular abelian core a II$_1$ factor and its non-separable analogue. We start by recalling some terminology involving von Neumann algebras. 
We will always work with tracial von Neumann algebras, i.e., von Neumann algebras $M$ endowed with a fixed faithful normal trace $\tau$. We endow $M$ with the $2$-norm given by $\|x\|_2=\tau(x^*x)^{1/2}$ and denote by $\mathcal U(M)$ its group of unitaries and by $(M)_1=\{x\in M\mid \|x\|\leq 1\}$ its (uniform) unit ball. We assume that all von Neumann subalgebras are unital.
For a von Neumann  subalgebra $A\subset M$, we denote by $\text{E}_A:M\rightarrow M$ the conditional expectation onto $A$ and by $\mathcal N_M(A)=\{u\in\mathcal U(M)\mid uAu^*=A\}$ the normalizer of $A$ in $M$. We say that a von Neumann algebra $M$ is {\it purely non-separable} if $pMp$ is non-separable, for every nonzero projection $p\in M$.
\subsection{Interwining by bimodules}

We recall the intertwining by bimodules theory from \cite[Theorem 2.1 and Corollary 2.3]{Po03}. 

\begin{thm}[\cite{Po03}]\label{intertwining} Let $(M,\tau)$ be a tracial von Neumann algebra and $A\subset pMp,B\subset qMq$ be von Neumann subalgebras. Then the following conditions are equivalent.
\begin{enumerate}
\item There exist nonzero projections $p_0\in A, q_0\in B$, a $*$-homomorphism $\theta\colon p_0Ap_0\rightarrow q_0Qq_0$  and a nonzero partial isometry $v\in q_0M p_0$ such that $\theta(x)v=vx$, for all $x\in p_0Ap_0$.

\item There is no net $u_n\in\mathcal U(A)$ satisfying $\|\emph{E}_B(x^*u_ny)\|_2\rightarrow 0$, for all $x,y\in pM$.
\end{enumerate}
\end{thm}

If (1) or (2) hold true,  we write $A\prec_{M}B$ and say that {\it a corner of $A$ embeds into $B$ inside $M$}.
If $Ap'\prec_{M}B$, for any nonzero projection $p'\in A\cap pMp$, we write $A\prec^{\text f}_{M}B$.

\subsection{Singular MASAs}
Let $(M,\tau)$ be a tracial von Neumann algebra. An abelian von Neumann subalgebra $A\subset M$ is called a {\it MASA} if it is maximal abelian and {\it singular} if it satisfies $\mathcal N_M(A)=\mathcal U(A)$ \cite{Dixmier54}. Note that a singular abelian von Neumann subalgebra $A\subset M$ is automatically a MASA.  

Two MASAs $A\subset pMp,B\subset qMq$ are called {\it disjoint} if $A\nprec_MB$.
The following result from \cite[Theorem A.1]{Po01} shows that disjointness for MASAs is the same as having no unitarily conjugated corners. In particular, disjointness of MASAs is a symmetric relation.

\begin{thm}[\cite{Po01}]\label{masa}
    Let $(M,\tau)$ be a tracial von Neumann algebra and $A\subset pMp,B\subset qMq$ be MASAs. Then $A\prec_MB$ if and only if $B\prec_MA$ and if and only if there exist nonzero projections $p_0\in A,q_0\in B$  such that $u(Ap_0)u^*=Bq_0$, for some $u\in\mathcal U(M)$.
\end{thm}

\subsection{The singular abelian core} We are now ready to give the following:

\begin{defn}\label{folded}
    Let $(M,\tau)$ be a tracial von Neumann algebra. We denote by $\mathcal S(M)$ the set of all families $\mathcal A=\{A_i\}_{i\in I}$, where $p_i \in M$ is a projection, $A_i\subset p_iMp_i$ is a singular MASA, for every $i\in I$, and $A_i,A_{i'}$ are disjoint, for every $i,i'\in I$ with $i\not=i'$. We denote $\text{d}(\mathcal A)=\sum_{i\in I}\tau(p_i)$, the {\em size} of the family $\cA$.
    Given $\mathcal A=\{A_i\}_{i\in I},\mathcal B=\{B_j\}_{j\in J}\in\mathcal S(M)$ we write $\mathcal A\leq \mathcal B$ if for every $i\in I$ and nonzero projection $p\in A_i$, there exists $j\in J$ such that $A_ip\prec_MB_j$. We say that $\mathcal A$ and $\mathcal B$ are {\it equivalent} and write $\mathcal A\sim \mathcal B$ if $\mathcal A\leq\mathcal B$ and $\mathcal B\leq\mathcal A$.
\end{defn}

\begin{lem}\label{max}
Let $(M,\tau)$ be a tracial von Neumann algebra. Then $\mathcal S(M)$ admits a maximal element with respect $\leq$. Moreover, any two maximal elements of $\mathcal S(M)$ with respect to $\leq$ are equivalent.
\end{lem}

\begin{proof}
    Let $\mathcal A=\{A_i\}_{i\in I}\in\mathcal S(M)$ be a maximal family with respect to inclusion. Then $\mathcal A$ is maximal with respect to $\leq$. To see this, let $\mathcal B=\{B_j\}_{j\in J}\in\mathcal S(M)$. If $\mathcal B\not\leq A$, then there are $j\in J$ and a nonzero projection $q\in B_j$ with $B_jq\nprec_MA_i$, for every $i\in I$. As $B_jq\subset qMq$ is a singular MASA, we get that $\mathcal A\cup\{B_jq\}\in\mathcal S(M)$, contradicting the maximality of $\mathcal A$ with respect to inclusion. The moreover assertion follows.
\end{proof}

\begin{defn}
    Let $(M,\tau)$ be a tracial von Neumann algebra. We denote by $\mathcal A_M$ the equivalence class consisting of all maximal elements of $\mathcal S(M)$ with respect to $\leq$, and call it the {\it singular abelian core} of $M$. We define the {\it rank} $\text{r}(M)$ of $M$ as the size, $\text{d}(\mathcal A)$, of  any $\mathcal A\in\mathcal A_M$. 
Note that $\text{r}(M)$ is a well-defined isomorphism invariant of $M$ since the map $\mathcal A\mapsto\text{d}(\mathcal A)$ is constant on equivalence classes.
    \end{defn}

    \begin{rem}\label{unfolded}
Definition \ref{folded} presents  the {\it folded form} of $\mathcal S(M)$, for a tracial von Neumann algebra $(M,\tau)$. Let $K$ be a large enough set, which contains the index set $I$ of any element $\mathcal A=\{A_i\}_{i\in I}$ of $\mathcal S(M)$. For instance, take $K$ to be the collection of all singular MASAs $A\subset pMp$, for all projections $p\in M$. We identify every  $\mathcal A=\{A_i\}_{i\in I}$ of $\mathcal S(M)$ with the singular abelian von Neumann subalgebra $\mathcal A=\oplus_{i\in I}A_i$  of $p\mathcal Mp$, where $\mathcal M=M\overline{\otimes}\mathbb B(\ell^2K)$ and $p=\oplus_{i\in I}p_i\in\mathcal M$.
This is the {\it unfolded form} of $\mathcal S(M)$. 
In this unfolded form, given $\mathcal A,\mathcal B\in\mathcal S(M)$, we have that $\mathcal A\leq\mathcal B$ (respectively, $\mathcal A\sim\mathcal B$) if and only if $\mathcal A\subset u\mathcal Bqu^*$ (respectively, $\mathcal A=u\mathcal Bu^*$), for a projection $q\in\mathcal B$ and unitary $u\in \mathcal M$. 

The unfolded form of the singular abelian core $\mathcal A_M$ of $M$ is then the unique (up to unitary conjugacy) singular abelian von Neumann subalgebra $\mathcal A\subset p\mathcal Mp$ generated by finite projections such that for any singular abelian von Neumann subalgebra $\mathcal B\subset q\mathcal Mq$, for a finite projection $q$, we have that $\mathcal B\prec_{\mathcal M}\mathcal A$. The rank $\text{r}(M)$ is then equal to  the semifinite trace, $(\tau\otimes\text{Tr})(p)$, of the unit  $p$ of $\mathcal A_M$. 
Notice that if the semifinite trace $(\tau\otimes\text{Tr})(p)$ of the support of $\mathcal A$ is infinite, then it can be viewed as a cardinality $\leq |K|$. We will in fact view  $\text{r}(M)$ this way, when infinite. 
 
\end{rem}

 \begin{rem}\label{d_M}  Let $M$ be an arbitrary separable II$_1$ factor.
       By a result in \cite{Po83},  $M$ admits a singular MASA. This result was strengthened  in \cite[Theorem 1.1]{Po16} where it was shown that $M$ contains an uncountable family of pairwise disjoint singular MASAs. Consequently, $\text{r}(M)> \aleph_0$. More recently, it was shown in  \cite[Theorem 1.1]{Sorin_Duke} that $M$ contains a copy of the hyperfinite II$_1$ factor $R\subset M$ which is coarse, i.e., such the $R$-bimodule $\text{L}^2(M)\ominus\text{L}^2(R)$ is a multiple of the coarse $R$-bimodule $\text{L}^2(R)\overline{\otimes}\text{L}^2(R)$. In combination with \cite[Proposition 2.6.3]{Sorin_Duke} and  \cite[Theorem 5.1.1]{SorinKS}, this implies that $M$ has 
       a continuous family of disjoint singular MASAs. Since the set of distinct self-adjoint elements in a separable II$_1$ factor has continuous cardinality $\frak c=2^{\aleph_0}$ and each singular MASA is generated by a self-adjoint element, it follows that $\text{r}(M)=\frak c$, for every separable II$_1$ factor $M$.

\end{rem}

\subsection{The singular abelian non-separable core} Remark \ref{d_M} shows that  the rank $\text{r}(M)$ is equal to the continuous cardinality $\frak c$ for any separable II$_1$ factor $M$, and thus cannot be used to distinguish such factors up to isomorphism. In contrast, we define in this section a non-separable analogue of $\text{r}(M)$, which will later enable us to prove the non-isomorphisms asserted by Theorem \ref{main}. 

\begin{defn}
    Let $(M,\tau)$ be a tracial von Neumann algebra. We say that a  von Neumann subalgebra $A\subset pMp$ is a {\it sans}-subalgebra of $M$ if it is singular abelian in $pMp$ and purely non-separable.  We denote by $\mathcal S_{\text{ns}}(M)\subset\mathcal S(M)$ the set of $\mathcal A=\{A_i\}_{i\in I}\in\mathcal S(M)$ such that $A_i$ is a sans-subalgebra, for every $i\in I$. We call any $\mathcal A\in\cS_{\text{ns}}(M)$ a {\em sans family} in $M$.
\end{defn}

Since Lemma \ref{max} trivially holds true if we replace $\mathcal S(M)$ by $\mathcal S_{\text{ns}}(M)$, we can further define:

\begin{defn}
    Let $(M,\tau)$ be a tracial von Neumann algebra. We denote by $\mathcal A_M^{\text{ns}}$ the equivalence class consisting of all maximal elements of $\mathcal S_{\text{ns}}(M)$ with respect to $\leq$, and call it the  {\it singular abelian non-separable core} (abbreviated, the {\it sans-core}) of $M$. We define the {\it sans-rank} $\text{r}_{\text{ns}}(M)$ of $M$ as the size, $\text{d}(\mathcal A)$, of  any $\mathcal A\in\mathcal A_M^{\text{ns}}$.
    \end{defn}

\begin{rem}\label{unfolded2} Like in Remark \ref{unfolded}, consider  $\mathcal M=M\overline{\otimes}\mathbb B(\ell^2K)$, for a large enough set $K$.
In the unfolded form of $\mathcal S_{\text{ns}}(M)$, the sans-core $\mathcal A_M^{\text{ns}}$ of $M$ is the unique (up to unitary conjugacy) sans-subalgebra $\mathcal A\subset p\mathcal Mp$ generated by finite projections such that for any sans-subalgebra $\mathcal B\subset q\mathcal Mq$, for a finite projection $q$, we have that $\mathcal B\prec_{\mathcal M}\mathcal A$. The sans-rank $\text{r}_{\text{ns}}(M)$ is then  the semifinite trace, $(\tau\otimes\text{Tr})(p)$, of the unit  $p$ of $\mathcal A_M^{\text{ns}}$. Like in Remark 2.6, when the semifinite trace of the support of the sans-core in this unfolded form is infinite, then we will view $\text{r}_{\text{ns}}(M)$ as a  cardinality $\leq |K|$. 
    \end{rem}

\begin{rem} 
    If $M$ is a separable II$_1$ factor, then we clearly have $\text{r}_{\text{ns}}(M)=0$. If $A\subset M$ is a singular MASA and $\omega$ is a free ultrafilter on $\mathbb N$, then $A^{\omega}\subset M^\omega$ is a purely non-separable singular MASA, see \cite[5.3]{Po83}. Moreover, disjoint MASAs in $M$ give rise to disjoint ultrapower MASAs in $M^\omega$.  
   By using these facts and results from \cite{SorinKS,Sorin_Duke} as in Remark \ref{d_M} we get that $\text{r}_{\text{ns}}(M^\omega)\geq \frak c$, for every separable II$_1$ factor $M$. But getting $\text{r}_{\text{ns}}(M^\omega)\leq \frak c$ is problematic, as besides the family of disjoint ultraproduct singular MASAs in $M^\omega$, which has cardinality $\frak c$, one may have singular MASAs that are not of this form. 
\end{rem}

The expression of $\text{r}_{\text{ns}}(M)$ as the semifinite trace of the support of the sans-core in unfolded form, as in Remark 2.10, implies the following scaling formula for $\text{r}_{\text{ns}}(M)$. We include below an alternative short proof using the folded form of $\mathcal S_{\text{ns}}(M)$.

\begin{prop}\label{fundamental group}
\label{main3}
Let $M$ be any II$_1$ factor and $t \in \R_+^*$. 
Then we have
\[\nsrk(M^t) = \nsrk(M)/t.\]
In particular, if $0 < \nsrk(M) < \infty$, then $M$ has trivial fundamental group, $\mathcal F(M)=\{1\}$.
\end{prop}

\begin{proof} It is enough to argue that $\text{r}_{\text{ns}}(qMq)=\text{r}_{\text{ns}}(M)/\tau(q)$, for every nonzero projection $q\in M$.
This follows immediately by using the fact that any $\mathcal A=\{A_i\}_{i\in I}\in\mathcal S(M)$ is equivalent to some $\mathcal B=\{B_j\}_{j\in J}\in\mathcal S(M)$, such that $B_j\subset q_jMq_j$, for some $q_j\leq q$, for every $j\in J$.   
\end{proof}

\section{Main results}\label{mainresults}

\subsection{Main technical result} This subsection is devoted to proving our main technical result. Throughout the subsection we use the following notation. Let $(M_1,\tau_1)$ and $(M_2,\tau_2)$ tracial von Neumann algebras and denote by $M=M_1* M_2$ their free product with its canonical trace $\tau$.

\begin{thm}\label{commutant}
 Let $P\subset pMp$ be a von Neumann subalgebra such that $P'\cap pMp$ is non-separable. Then $P\prec_MM_1$ or $P\prec_MM_2$.   
\end{thm}

The proof of Theorem \ref{commutant} is based on the main technical result of \cite{IPP05}.
 By \cite[Section 5.1]{PV09}, given $\rho\in (0,1)$, we have a unital tracial completely positive map $\text{m}_\rho:M\rightarrow M$ such that  $\text{m}_\rho(x_1x_2\cdots x_n)=\rho^nx_1x_2\cdots x_n$, for every $n\in\mathbb N$ and $x_i\in M_{i_j}\ominus \mathbb C1$, where $i_j\in\{1,2\}$, for every $1\leq j\leq n$, and $i_j\not=i_{j+1}$, for every $1\leq j\leq n-1$. Note that $\lim\limits_{\rho\rightarrow 1}\|\text{m}_\rho(x)-x\|_2=0$ and the map $(0,1)\ni\rho\mapsto\|\text{m}_{\rho}(x)\|_2$ is increasing, for every $x\in M$.
 The implication (1) $\Rightarrow$ (2) follows from \cite[Theorem 4.3]{IPP05}, formulated here as in \cite[Theorem 5.4]{PV09}, see also \cite[Section 5]{Ho07}. 

\begin{thm}[\cite{IPP05}]\label{IPP}
 Let $P\subset pMp$ be a von Neumann subalgebra. 
 Then the following two conditions are equivalent:
 \begin{enumerate}
     \item There exists $\rho\in (0,1)$ such that 
     $\inf_{u\in\mathcal U(P)}\|\emph{m}_\rho(u)\|_2>0$.
     \item $P\prec_MM_1$ or $P\prec_MM_2$.
 \end{enumerate}
\end{thm}

\begin{proof}
   Assume that (1) holds. Since $\tau(x^*\text{m}_{\rho^2}(x))=\|\text{m}_\rho(x)\|_2^2$, for every $x\in M$, we get that 
   $\inf_{u\in\mathcal U(P)}\tau(u^*\text{m}_{\rho^2}(u))>0$ and \cite[Theorem 5.4]{PV09} implies (2). 
   
   To see that (2) $\Rightarrow$ (1), assume that $P\prec_MM_i$, for some $i\in\{1,2\}$. By  Theorem \ref{intertwining}
  we find a nonzero partial isometry $v\in M$ such that $v^*v=p_0p$, for some projections $p_0\in P,p'\in P'\cap pMp$, and $(p_0Pp_0)_1p'\subset v^*(M_i)_1v$. Since $\|\text{m}_\rho(x)-x\|_2\leq |\rho-1|$, for every $x\in (M_i)_1$, we get that $\lim_{\rho\rightarrow 1}(\sup_{x\in (p_0Pp_0)_1p'}\|\text{m}_\rho(x)-x\|_2)=0$. Let $p_1$ be the central support of $p_0$ in $P$ and denote $p''=p_1p'\in P'\cap pMp$. It follows that $\lim_{\rho\rightarrow 1}(\sup_{x\in (Pp'')_1}\|\text{m}_\rho(x)-x\|_2)=0$. 
  From this it is easy to deduce that 
   $\liminf_{\rho\rightarrow 1}(\inf_{u\in\mathcal U(P)}\|\text{m}_\rho(u)\|_2)\geq \|p''\|_2>0$, which clearly implies (1).
\end{proof}

\begin{cor}\label{sepsubalg}
    Let $P\subset pMp$ be a von Neumann subalgebra such that $P\nprec_MM_1$ and $P\nprec_MM_2$. Then there exists a separable von Neumann subalgebra $Q\subset P$ such that $Q\nprec_MM_1$ and $Q\nprec_MM_2$.
\end{cor}

\begin{proof}
   Since $P\nprec_MM_1$ and $P\nprec_MM_2$, by Theorem \ref{IPP} we find a sequence $u_n\in\mathcal U(P)$ such that $\|m_{1-1/n}(u_n)\|_2\leq 1/n$. Let $Q\subset P$ be the separable von Neumann subalgebra generated by $\{u_n\}_{n\geq 1}$.  Let $\rho\in (0,1)$. Then for every $n\geq 1$ such that $\rho\leq 1-1/n$ we have that $\|\text{m}_{\rho}(u_n)\|_2\leq\|\text{m}_{1-1/n}(u_n)\|_2\leq 1/n$.
   This implies $\inf_{u\in\mathcal U(Q)}\|\text{m}_\rho(u)\|_2=0$.  Since this holds for every $\rho\in (0,1)$, Theorem \ref{IPP} implies that  $Q\nprec_MM_1$ and $Q\nprec_MM_2$.
\end{proof}

\begin{lem}\label{sep}
Let $Q\subset M$ be a separable von Neumann subalgebra. Then we can find separable von Neumann subalgebras $N_1\subset M_1$ and $N_2\subset M_2$ such that $Q\subset N_1*N_2$.
\end{lem}

\begin{proof}
For $i\in\{1,2\}$ let $\mathcal B_i$ be an orthonormal basis of $\text{L}^2(M_i)\ominus\mathbb C1$ such that $\mathcal B_i\subset M_i\ominus \mathbb C1$. Let $\mathcal B_0$ be the set of $\xi_1\xi_2\cdots\xi_n$, where $n\in\mathbb N$,  $\xi_i\in \mathcal B_{i_j}$, for some $i_j\in\{1,2\}$, for every $1\leq j\leq n$, and $i_j\not=i_{j+1}$, for every $1\leq j\leq n-1$. Then $\mathcal B=\mathcal B_0\cup\{1\}$ is an orthonormal basis of $\text{L}^2(M)$. 

Let $\{x_k\}_{k\geq 1}$ be a sequence which generates $Q$.
Then $\mathcal C=\cup_{k\geq 1}\{\xi\in\mathcal B\mid \langle x_k,\xi\rangle\not=0\}$ is countable. For $i\in\{1,2\}$, let $\mathcal C_i$ be the countable set of all $\xi\in\mathcal B_i$ which appear in the decomposition of some element of $\mathcal C$. The von Neumann subalgebra $N_i$ of $M_i$ generated by $\mathcal C_i$ is separable, for every $i\in\{1,2\}$. Since by construction we have that $Q\subset N_1*N_2$, this finishes the proof.
\end{proof}

\begin{proof}[Proof of Theorem \ref{commutant}]
    Assume by contradiction that $P\nprec_MM_1$ and $P\nprec_MM_2$. By applying Corollary \ref{sepsubalg}, we can find a separable von Neumann subalgebra $Q\subset P$ such that $Q\nprec_MM_1$ and $Q\nprec_MM_2$. By Lemma \ref{sep}, we can further find separable von Neumann subalgebras $N_1\subset M_1$ and $N_2\subset M_2$, such that $Q\subset N:=N_1*N_2$. Denote $R=M_1*N_2$.

    Since $Q\nprec_MM_1$, $Q\subset R\subset M$ and $N_1\subset M_1$, we get that $Q\nprec_R{N_1}$. Since $Q\subset N$ and $R=M_1*_{N_1}N$, \cite[Theorem 1.1]{IPP05} implies that $Q'\cap R=Q'\cap N$. Next, since $Q\nprec_MM_2$ and $N_2\subset M_2$, we get that $Q\nprec_MN_2$. Since $Q\subset R$ and $M=R*_{N_2}M_2$, applying \cite[Theorem 1.1]{IPP05} again gives that $Q'\cap M=Q'\cap R$. Altogether, we get that $Q'\cap M=Q'\cap N$. Since $N$ and thus $Q'\cap N$ is separable,  using that $P'\cap M\subset Q'\cap M$, we conclude that $P'\cap M$ is separable.
\end{proof}

\subsection{Non-separable MASAs in free product algebras} In this subsection, we derive some consequences of Theorem \ref{commutant} to the structure of non-separable MASAs in free product algebras.

\begin{cor}
    \label{cut}
Let $(M_1,\tau_1)$ and $(M_2,\tau_2)$ be tracial von Neumann algebras, and denote by $M=M_1*M_2$ their free product. 
       Let $A\subset pMp$ be a purely non-separable MASA. Then there exist projections $(p_k)_{k\in K}\subset A$ and unitaries $(u_k)_{k\in K}\subset M$
 such that $\sum_{k\in K}p_k=p$ and for every $k\in K$,   $u_kAp_ku_k^*\subset M_i$, for some $i\in\{1,2\}$.    
 
\end{cor}
\begin{proof}  By a maximality argument, it suffices to prove that if $q\in A$ is a nonzero projection, then there are a nonzero projection $r\in Aq$, a unitary $u\in M$ and $i\in\{1,2\}$ such that $uAru^*\subset M_i$. 

    To this end, let $q\in A$ be a nonzero projection. Since $(Aq)'\cap qMq=Aq$ is non-separable, Theorem \ref{commutant} implies that there is $i\in\{1,2\}$ such that $Aq\prec_MM_i$. By Theorem \ref{intertwining}, we can find nonzero projections $e\in Aq,f\in M_i$, a nonzero partial isometry $v\in fMe$ and a $*$-homomorphism $\theta:Ae\rightarrow fM_if$ such that $\theta(x)v=vx$, for every $x\in Ae$. Then $r:=v^*v\in (Ae)'\cap eMe=Ae$ and $vv^*\in\theta(Ae)'\cap fMf$. Since $\theta(Ae)\subset fM_if$ is diffuse, by applying \cite[Theorem 1.1]{IPP05} (see also \cite[Remarks 6.3.2)]{PoJOT}) we get that $vv^*\in fM_if$. 
    Finally, let $u\in M$ be any unitary such that $ur=v$. Then $uAru^*=vArv^*=vAev^*=\theta(Ae)vv^*\subset M_i$, which finishes the proof.
 \end{proof}

We continue by generalizing Corollary \ref{cut} to arbitrary tracial free products. 

\begin{cor}\label{arbitrary free products}
    Let $(M_i,\tau_i)$, $i\in I$, be a colection of tracial von Neumann algebras, and denote by $M=*_{i\in I}M_i$ their free product.   Let $A\subset pMp$ be a purely non-separable MASA. Then there exist projections $(p_k)_{k\in K}\subset A$ and unitaries $(u_k)_{k\in K}\subset M$
 such that $\sum_{k\in K}p_k=p$ and for every $k\in K$,   $u_kAp_ku_k^*\subset M_i$, for some $i\in I$.   
\end{cor}

\begin{proof}
Let $A_0\subset A$ be a separable diffuse von Neumann subalgebra. Reasoning similarly to the proof of Lemma \ref{sep} yields a countable set $J\subset I$ such that $A_0\subset *_{j\in J}M_j$.
    Since $A_0$ is diffuse, \cite[Theorem 1.1]{IPP05} gives that $A\subset A_0'\cap pMp\subset *_{j\in J}M_j$. Thus, in order to prove the conclusion, after replacing $I$ with $J$, we may take  $I$ countable. Enumerate $I=\{i_m\}_{m\geq 1}$.

    Let $\{p_k\}_{k\in K}\subset A$ be a maximal family, with respect to inclusion, of pairwise orthogonal projections such that for every $k\in K$, there are a unitary $u_k\in M$ and $i\in I$ such that $u_kAp_ku_k^*\subset M_i$. In order to prove the conclusion it suffices to argue that $\sum_{k\in K}p_k=p$. Put $r:=p-(\sum_{k\in K}p_k)$.
    
    Assume by contradiction that $r\not=0$. 
 We claim that 
 \begin{equation}\label{tail}\text{$Ar\nprec_M*_{m\leq n}M_{i_m}$, for every $n\geq 1$.}\end{equation} 
 Otherwise, if \eqref{tail} fails for some $n\geq 1$, then the proof of Corollary \ref{cut} gives a nonzero projection $s\in Ar$ and a unitary $u\in M$ such that $uAsu^*\subset *_{m\leq n}M_{i_m}.$ Applying Corollary \ref{cut} repeatedly gives a nonzero projection $t\in As$ and a unitary $v\in *_{m\leq n}M_{i_m}$ such that $vuAtu^*v^*\subset M_{i_m}$, for some $1\leq m\leq n$.
This contradicts the maximality of the family $\{p_k\}_{k\in K}$, and proves \eqref{tail}.
 
 If $e\in (Ar)'\cap rMr=Ar$ is a nonzero projection, then $(Ae)'\cap eMe=Ae$ is nonseparable. Since $Ae\nprec_M*_{m\leq n}M_{i_m}$ by \eqref{tail}, Theorem \ref{sep} implies that $Ae\prec_M*_{m>n}M_{i_m}$ and thus \begin{equation}\label{Ar}\text{$Ar\prec_M^{\text{f}}*_{m>n}M_{i_m}$, for every $n\geq 1$.}\end{equation}

 To get a contradiction, we follow the proof of \cite[Proposition 4.2]{HU15}. Let $\widetilde M=M*M$, identify $M$ with  $M*1\subset\widetilde M$, and denote by $\theta$ the free flip automorphism of $\widetilde M$. Endow $\mathcal H=\text{L}^2(\widetilde M)$ with the $M$-bimodule structure given by $x\cdot \xi \cdot y=\theta(x)\xi y$, for every $x,y\in M$ and $\xi\in\mathcal H$. 
Using \eqref{Ar}, 
the proof of \cite[Proposition 4.2]{HU15} yields a sequence of vectors $\eta_n\in r\cdot\mathcal H\cdot r$ such that 
$\|\eta_n\|_2\rightarrow\|r\|_2$, $\|x\cdot\eta_n\|_2\leq\|x\|_2$ and $\|a\cdot\eta_n-\eta_n\cdot a\|_2\rightarrow 0$, for every $x\in rMr$ and $a\in Ar$.

Since the $Ar$-bimodule $r\cdot H\cdot r$ is isomorphic to a multiple of the coarse $Ar$-bimodule, we obtain a sequence of vectors $\zeta_n\in\oplus_{\mathbb N}(\text{L}^2(Ar)\otimes\text{L}^2(Ar))$ such that  $\|\zeta_n\|_2\rightarrow\|r\|_2$, $\|a\cdot\zeta_n\|_2\leq\|a\|_2$ and $\|a\cdot\zeta_n-\zeta_n\cdot a\|_2\rightarrow 0$, for every $a\in Ar$. By reasoning similarly to the proof of Lemma \ref{sep}, we find a separable von Neumann subalgebra $A_0\subset Ar$ such that $\zeta_n\in\oplus_{\mathbb N}(\text{L}^2(A_0)\otimes\text{L}^2(A_0))$. 

As $A_0$ is separable and $Aq$ is non-separable, for every nonzero projection $q\in Ar$,  $Ar\nprec_{Ar}A_0$. Theorem \ref{intertwining} gives a unitary $u\in Ar$ with $\|\text{E}_{A_0}(u)\|_2\leq\|r\|_2/2$. Put $a=u-\text{E}_{A_0}(u)\in A$. 
Since $a\cdot\zeta_n\in\oplus_{\mathbb N}((\text{L}^2(Ar)\ominus \text{L}^2(A_0))\otimes\text{L}^2(A_0))$ and $\zeta_n\cdot a\in\oplus_{\mathbb N} (\text{L}^2(A_0)\otimes(\text{L}^2(Ar)\ominus\text{L}^2(A_0))$, we have that $\langle a\cdot\zeta_n,\zeta_n\cdot a\rangle=0$, for every $n$.
Using that $\|a\cdot\zeta_n-\zeta_n\cdot a\|_2\rightarrow 0$, we get that $\|a\cdot\zeta_n\|_2\rightarrow 0$. On the other hand, $\|a\cdot\zeta_n\|_2\geq \|u\cdot\zeta_n\|_2-\|\text{E}_{A_0}(u)\cdot\zeta_n\|_2\geq \|\zeta_n\|_2-\|\text{E}_{A_0}(u)\|_2\geq \|\zeta_n\|_2-\|r\|_2/2.$ Since $\|\zeta_n\|_2\rightarrow\|r\|_2>0$,  we altogether get a contradiction, which finishes the proof.
\end{proof}

We end this subsection by noticing that in the case $A\subset pMp$ is a singular MASA and $M_i$ is abelian, for every $i\in I$, the conclusion of Corollay \ref{arbitrary free products} can be strengthened as follows:

\begin{cor}\label{corner_identification}
    In the context of Corollary \ref{arbitrary free products}, assume additionally that $A\subset pMp$ is singular and $M_i$ is abelian, for every $i\in I$. Then there exist projections $(q_i)_{i\in I}\subset A$ and unitaries $(v_i)_{i\in I}\subset M$ such that $\sum_{i\in I}q_i=p$, $e_i=v_iq_iv_i^*\in M_i$ and $v_iAq_iv_i^*= M_ie_i$, for every $i\in I$.
\end{cor}

\begin{proof}
    By applying Corollary \ref{arbitrary free products} we find projections $(p_k)_{k\in K}\subset A$ and unitaries $(u_k)_{k\in K}\subset M$ such that $\sum_{k\in K}p_k=p$ and for every $k\in K$, $u_kAp_ku_k^*\subset M_{i_k}$, for some $i_k\in I$. Let $k\in K$ and put $r_k:=u_kp_ku_k^*\in M_{i_k}$. Since $u_kAp_ku_k^*\subset r_kMr_k$ is a MASA and $M_{i_k}$ is abelian we deduce that $u_kAp_ku_k^*=M_{i_k}r_k$, for every $k\in K$.
Let $k,k'\in K$ such that $k\not=k'$ and $i_k=i_{k'}$.
Since $A\subset pMp$ is singular and $p_kp_{k'}=0$, there are no nonzero projections $s\in Ap_k,s'\in Ap_{k'}$ such that $As$ and $As'$ are unitarily conjugated in $M$. This implies that $r_kr_{k'}=0$.
  Using this fact, it follows that if we denote $q_i=\sum_{k\in K,i_k=i}p_k$, then $v_iAq_iv_i^*\subset M_i$, for every $i\in I$. For $i\in I$, let $e_i=v_iq_iv_i^*\in M_i$. Then $v_iAq_iv_i^*\subset M_ie_i$ and since $v_iAq_iv_i^*\subset M_ie_i$ is a MASA, while $M_ie_i$ is abelian, it follows that $v_iAq_iv_i^*=M_ie_i$, as claimed.
\end{proof}

\subsection{The non-separable rank of free product von Neumann algebras}
In this section, we show that the sans core of a free product of tracial von Neumann algebras $M=*_{i\in I} M_i$ is the union  of the sans cores of $M_i, i\in I$. This allows us to deduce that the sans  rank  of $M$ is the sum of the sans ranks of $M_i,i\in I$. 

\begin{thm}
\label{fp sum}
 Let $(M_i,\tau_i)$, $i\in I$, be a colection of tracial von Neumann algebras, and denote by $M=*_{i\in I}M_i$ their free product. 
Then $\emph{r}_{\emph{ns}}(M) = \sum_{i\in I}\emph{r}_{\emph{ns}}(M_i)$.
Moreover, if $\mathcal A_i\in\mathcal A_{M_i}^{\emph{ns}}$, for every $i\in I$, then $\cup_{i\in I}\mathcal A_i\in \mathcal A_{M}^{\emph{ns}}$.
\end{thm}
The moreover assertion uses implicitly the fact, explained in the proof, that every sans family in $M_i$ is naturally a sans family in $M$, for every $i\in I$.
\begin{proof}
    We have two inequalities to prove.
    
    {\bf Inequality 1.} $\text{r}_{\text{ns}}(M) \geq \sum_{i\in I}\text{r}_{\text{ns}}(M_i)$.

    This inequality relies on several facts on free products, all of which follow from \cite[Theorem 1.1]{IPP05}. Let $i,j\in I$ with $i\not=j$.
    
    \begin{enumerate}
    \item If $A\subset pM_ip$ is a MASA, then  $A\subset pMp$ is a MASA. 
    \item If $A \subset pM_ip$ is a singular diffuse von Neumann subalgebra, then $A\subset pMp$ is singular.
    \item If $A \subset pM_ip$, $B \subset q M_iq$ are von Neumann subalgebras with $A \prec_M B$, then $A \prec_{M_i} B$.
    \item If $A \subset pM_ip$ and $B \subset qM_jq$ are diffuse von Neumann subalgebras, then $A \nprec_M B$.
    \end{enumerate}

    For $i\in I$, let $\mathcal A_i\in\mathcal A_{M_i}^{\text{ns}}$ be a maximal sans family in $M_i$. We view every (not necessarily unital) subalgebra of $M_i$ as a subalgebra of $M$. Then facts (1)-(3) imply that  $\mathcal A_i$ is a sans family in $M$. Moreover, fact (4) implies that $\mathcal A:=\cup_{i\in I}\mathcal A_i$ is a sans family in $M$.  Thus,  $$\text{r}_{\text{ns}}(M)\geq \text{d}(\mathcal A)=\sum_{i\in I}\text{d}(\mathcal A_i)=\sum_{i\in I}\text{r}_{\text{ns}}(M_i).$$

    {\bf Inequality 2.} $\nsrk(M) \leq \sum_{i\in I}\nsrk(M_i)$.
    
    Let $\cA=\{A_l\}_{l\in L} \in \mathcal A^{\text{ns}}_M$ be a maximal sans family in $M$. Let $l\in L$.
   Applying Corollary \ref{arbitrary free products} to $A_l$ gives projections $(p_{k,l})_{k\in K_l}$ and unitaries $(u_{k,l})_{k\in K_l}$ such that for every $k\in K_l$ we have $u_{k,l}A_lp_{k,l}u_{k,l}^*\subset M_i$, for some $i\in I$. For $i\in I$, let $\mathcal A_i\in\mathcal S_{\text{ns}}(M_i)$ be the collection of sans-subalgebras of $M_i$ of the form $u_{k,l}A_lp_{k,l}u_{k,l}^*$, for all $l\in L,k\in K_l$ such that $u_{k,l}A_lp_{k,l}u_{k,l}^*\subset M_i$. Then $\mathcal A$ is equivalent to $\cup_{i\in I}\mathcal A_i$, which allows us to conclude that $$\text{r}_{\text{ns}}(M)=\text{d}(\mathcal A)=\sum_{i\in I}\text{d}(\mathcal A_i)\leq\sum_{i\in I} \text{r}_{\text{ns}}(M_i).$$ This finishes the proof of the main assertion. 
    The moreover assertion now follows by combining the proofs of inequalities 1 and 2.
\end{proof}

\subsection{Proof of Theorem \ref{main}}
In preparation for the proof of Theorem \ref{main}, we first record the following direct consequence of Theorem \ref{fp sum}:

\begin{cor}\label{ab fp}
    Let $(A_i,\tau_i)$, $i\in I$, be a collection of diffuse tracial abelian von Neumann algebras, and denote by $M=*_{i\in I}A_i$ their free product. For $i\in I$, let $p_i\in A_i$ be the maximal (possibly zero) projection such that $A_ip_i$ is purely non-separable. Then $\emph{r}_{\emph{ns}}(M)=\sum_{i\in I}\tau_i(p_i)$. Moreover, if $|I|\geq 2$ and $\sum_{i\in I}\tau_i(p_i)\in (0,+\infty)$, then $M$ is a II$_1$ factor with $\mathcal F(M)=\{1\}$. 
    Also, the sans-core of $M$ is given by $\mathcal A_M^{ns}=\{A_ip_i\}_{i\in I}.$
\end{cor}

\begin{proof}
    Let $i\in I$. Since $\{A_ip_i\}\in\mathcal S_{\text{ns}}(A_i)$ is a maximal element, we get that $\text{r}_{\text{ns}}(A_i)=\tau_i(p_i)$. The assertions now follow by using Theorem \ref{fp sum}, Proposition \ref{fundamental group}, and the fact that any free product of diffuse tracial von Neumann algebras is a II$_1$ factor.
\end{proof}

\begin{proof}[Proof of Theorem \ref{main}]
    Let $(A,\tau)$ be a diffuse non-separable tracial abelian von Neumann algebra. Let $p\in A$ be the maximal, necessarily non-zero, projection such that $Ap$ is purely non-separable. 
    By Corollary \ref{ab fp}, $\text{r}_{\text{ns}}(A^{*n})=n\tau(p)$, for every $2\leq n\leq\infty$. Since $p\not=0$, we get that $A^{*n}$, $2\leq n\leq \infty$, are mutually non-isomorphic, and  $\mathcal F(A^{*n})=\{1\}$, for $2\leq n<\infty$.
\end{proof}

\section{Further remarks and open problems}\label{Remarks and open problems}

\subsection{Freely complemented maximal amenable MASAs in $A^{*n}$} 
The question of whether the II$_1$ factors $A^{*n}$, $2\leq n \leq \infty$, are non-isomorphic for a non-separable diffuse tracial abelian von Neumann algebra $A$ was asked in \cite{BP23}. This was motivated by the consideration of certain ``radial-like" von Neumann subalgebras of $M=A^{*n}$, for $2\leq n\leq\infty$. Specifically, for every $1\leq k\leq n$, let $s_k$ be a semicircular self-adjoint element belonging to $A_k$, the $k^{\text{th}}$ copy of $A$ in $M$. For an $\ell^2$-summable family of real numbers $t=(t_k)$ with at least two non-zero entries, denote by $A(t)$ the abelian von Neumann subalgebra of $M$ generated by $\sum_kt_ks_k$.
 It was shown in \cite{BP23} that $A(t)\subset M$ is maximal amenable and $A(t), A(t')$ are disjoint if $t$ and $t'$ are not proportional. A key point in proving this result was to show that $A(t)\not\prec_{M} A_k$, for every $k$. Since the MASAs $A(t)$ are separable, 
despite $A$ being non-separable, this suggested that the only way to obtain a purely non-separable MASA in $M$ is to ``re-pack'' pieces of $A_k$, $1\leq k\leq n$. This further suggested the possibility of recovering $n$ from the isomorphism class of $M$. 

The construction of the family of radial-like maximal amenable MASAs $A(t)\subset M$ in \cite{BP23} was triggered by an effort to obtain examples of non freely complemented maximal amenable MASAs in the free group factors $L\mathbb F_n$. However, this remained open (see though \cite[Remark 1.4]{BP23}  for further comments concerning the inclusions $A(t)\subset A^{*n}$). Thus, there are no known examples of non freely complemented maximal amenable von Neumann subalgebras of $L\mathbb F_n$. It may 
be that in fact any maximal amenable $B\subset L\mathbb F_n$ is freely complemented (a property/question which we abbreviate as {\it FC}), see \cite[Question 5.5]{Sorin_Duke} and the  introduction of \cite{BP23}.

A test case for the FC question is proposed 
in the last paragraph of \cite{Sorin_Duke}. There it is pointed out that 
if $\{B_i\}_i$ are diffuse amenable von Neumann subalgebras of $L\mathbb F_n$ with $B_i$ freely complemented 
and $B_i\nprec_{L\mathbb F_n}B_j$, for every $i\neq j$,    
then $B=\oplus_i u_ip_iB_ip_iu_i^*$ is maximal amenable in $M$ by \cite{Sorin_generatorMASA}, 
for any projections $p_i\in B_i$ and unitaries 
$u_i\in M$ satisfying $\sum_i u_ip_iu_i^*=1$.
Thus, if FC is to hold then $B$ should be freely complemented as well. 

The FC question is equally interesting for the factors $M=A^{*n}$ with $A$ purely non-separable abelian. If  $A_k$ denotes the $k^{\text{th}}$ copy of $A$ in $M$, for every $1\leq k \leq n$, then by 
Theorem \ref{fp sum}, 
any purely non-separable singular abelian $B\subset M$ is of the form 
$B=\sum_k u_k A_kp_k u_k^*$ for some projections 
$p_k\in A_k$ and unitaries $u_k\in M$ 
with $\sum_k u_kp_ku_k^*=1$. Thus, $B$ is maximal amenable by \cite{Sorin_generatorMASA}. Hence, if  FC is to hold, then Theorem \ref{fp sum} suggests that the free complement of $B$ could be obtained by  a  ``free reassembling'' of unitary conjugates of pieces of $\{A_k(1-p_k)\}_{k=1}^n$.

\subsection{On the calculation of symmetries of $A^{*n}$} Let $M=A^{*n}$ with $A$ purely non-separable abelian. Theorem \ref{fp sum} shows that if  $\theta\in \text{\rm Aut}(M)$ then $\theta(\mathcal A^{\text{\rm ns}}_M)=\mathcal A^{\text{\rm ns}}_M$, modulo the equivalence in $\mathcal S_{\text{\rm ns}}(M)$ defined in Subsection 2.4. This suggests that one could perhaps explicitly calculate $\text{\rm Out}(M)$, for instance by identifying it with the Tr-preserving automorphisms $\alpha$ of the sans-core $\mathcal A^{\text{\rm ns}}_M$, viewed  in its unfolded form. In order to obtain from  an arbitrary such $\alpha$ an automorphism $\theta_\alpha$ of $M$ it would be sufficient to solve the FC question in its ``free repacking'' form explained in Remark 4.1 above. To prove that such a map $\alpha \mapsto \theta_\alpha$ is surjective one would need to show that if 
$\theta\in \text{\rm Aut}(M)$ implements the identity on the sans-core 
$\mathcal A^{\text{\rm ns}}_M$, then $\theta$ is inner on $M$. 

This heuristic is supported by the case of automorphisms $\theta$ of the free group $\mathbb F_2$: if $\theta(a)=a$
and $\theta(b)=gbg^{-1}$, for some $g\in \mathbb F_2$, where $a,b$ denote 
the free generators of $\mathbb F_2$, then $g$ must be of the form $g=a^k$, and so $\theta=\text{\rm Ad}(g)$ is inner. 

However, this phenomenon fails for the free groups $\mathbb F_n$ on $n\geq 3$ generators. Specifically, any $e\neq g\in \mathbb F_{n-1}=\langle a_1,\cdots,a_{n-1}\rangle$ gives rise to an outer automorphism $\theta_g$ on $\mathbb F_n$ defined by $\theta_g(a_i)=a_i$, if $1\leq i \leq n-1$, and $\theta_g(a_n)=ga_ng^{-1}$, where $a_1,\cdots, a_n$
are the free generators of $\mathbb F_n$. Similarly, if $M=A_1* ... *A_n$, 
with $A_i$ abelian diffuse, and $n\geq 3$, then any non-scalar unitary $u\in A_1 * ... *A_{n-1} * 1$ gives rise to an outer automorphism $\theta_u$ of $M$ 
defined by $\theta_u(x)=x$, if $x\in A_1 *\cdots *A_{n-1} * 1$, and $\theta_u(x)=uxu^*$, if $x\in 1*A_n$. 

A related problem is to investigate the structure of irreducible 
subfactors of finite Jones index $N\subset M=A^{*n}$, for $A$ purely non-separable abelian, with an identification of the sans-core, the sans-rank of $N$ and of the set of possible indices $[M:N]$, in the spirit of \cite[Section 7]{Po01}.

\subsection{Amplifications of $A^{*n}$} 
While Theorem \ref{main} shows that $\mathcal F(A^{*n})=1$ if $A$ is non-separable abelian and $n\geq 2$ is finite, it is still of interest to identify the amplifications $(A^{*n})^t$, for $t>0$. For arbitrary $t$ this remains open, 
but for $t=1/k$, $k\in\mathbb N$, we have the following result. We are very grateful to Dima Shlyakhtenko for pointing out to us that the $1/2$-amplification of $A^{*n}$ can be explicitly calculated for arbitrary diffuse $A$ by using existing models in free probability, a fact that  stimulated us to investigate the general $1/k$ case.

\begin{prop}\label{free_product_abelian}
 Let $(A_i,\tau_i)$, $i\in I$, be a countable collection of diffuse tracial abelian von Neumann algebras. Put $M=*_{i\in I}A_i$ and assume that $|I|\geq 2$. Let $k\geq 2$ and for every $i\in I$, let $p_{i,1},\cdots,p_{i,k}\in A_i$ be projections such that $\tau(p_{i,j})=1/k$, for every $1\leq j\leq k$, and $\sum_{j=1}^kp_{i,j}=1$.  
 
 Then $M$ is a II$_1$ factor and $M^{1/k}\cong (*_{i\in I,1\leq j\leq k}A_ip_{i,j})*D$, where
 \begin{enumerate}
\item $D=L\mathbb F_{1+|I|k(k-1)-k^2}$, if $I$ is finite, and
 \item $D=\mathbb C1$, if $I$ is infinite.
 \end{enumerate}
 \end{prop}

Recall that the {\it interpolated free group factors},  $L\mathbb F_r$, $1<r\leq\infty$,   introduced in \cite{Ra94,Dy94}, satisfy the formulas  \begin{equation}\label{interpolated}\text{$L\mathbb F_r*L\mathbb F_{r'}\cong L\mathbb F_{r+r'}$\; and \; $(L\mathbb F_r)^t\cong L\mathbb F_{1+\frac{(r-1)}{t^2}}$,\; for every $1\leq r,r'\leq\infty$ and $t>0$.}\end{equation}

\begin{proof}
We will use  the following consequence of \cite[Theorem 1.2]{Dy93}:

\begin{fact}[\cite{Dy93}]\label{dykema}
Let $P,Q$ be two tracial von Neumann algebras, and $e\in P$ be a central projection (hence, $P=Pe\oplus P(1-e))$. Denote $R=P*Q$ and $S=(\mathbb Ce\oplus P(1-e))*Q\subset R$.
Then $Pe$ and $eSe$ are free and together generate $eRe$, hence $eRe\cong Pe*eSe$.
\end{fact}
Specifically, we will use the following consequence of Fact \ref{dykema}:

\begin{claim}\label{compress}
Let $P,Q$ be tracial von Neumann algebras and $k\geq 2$. Assume that $P$ and $Q$ admit projections $e_1,\cdots,e_k\in P$ and $f_1,\cdots,f_k\in Q$ such that $e_i$ is central in $P$, $\tau(e_i)=\tau(f_i)=1/k$, for every $1\leq i\leq k$, $\sum_{j=1}^ke_j=1$ and $\sum_{j=1}^kf_j=1$. Then $e_1(P*Q)e_1\cong Pe_1*\cdots*Pe_k*e_1((\mathbb Ce_1\oplus\cdots\oplus\mathbb Ce_k)*Q)e_1$.

\end{claim}

\begin{proof}[Proof of Claim \ref{compress}]
Note that $e_1$ is equivalent to $e_j$ in $(\mathbb Ce_1\oplus\dots\oplus\mathbb Ce_k)*(\mathbb Cf_1\oplus\cdots\oplus\mathbb Cf_k)$ and so in $(\mathbb Ce_1\oplus\dots\oplus\mathbb Ce_k)*Q$, for every $2\leq j\leq k$. This
follows from \cite[Remark 3.3]{Dy94} if $k=2$
and because $(\mathbb Ce_1\oplus\dots\oplus\mathbb Ce_k)*(\mathbb Cf_1\oplus\cdots\oplus\mathbb Cf_k)\cong\text{L}(\mathbb Z/k\mathbb Z*\mathbb Z/k\mathbb Z)$ is a II$_1$ factor if $k\geq 3$.

Denote $e_j'=1-\sum_{l=1}^je_l$ and $P_j=\mathbb Ce_1\oplus\cdots\oplus\mathbb Ce_j\oplus Pe_j'$, for every $1\leq j\leq k$. We claim that
\begin{equation}\label{freefactor}\text{$e_1(P*Q)e_1\cong Pe_1*\cdots*Pe_j*e_1(P_j*Q)e_1$, for every $1\leq j\leq k$}.\end{equation}
When $j=1$, $e_1'=1-e_1$ and thus equation \eqref{freefactor} follows from Fact \ref{dykema}. Assume that \eqref{freefactor} holds for some $1\leq j\leq k-1$.
Since $e_{j+1}\in P_j$ is a central projection, $P_je_{j+1}=Pe_{j+1}$ and $\mathbb Ce_{j+1}\oplus P_j(1-e_{j+1})=P_{j+1}$,  Fact \ref{dykema} gives that $e_{j+1}(P_j*Q)e_{j+1}\cong Pe_{j+1}*e_{j+1}(P_{j+1}*Q)e_{j+1}$. 
The observation made in the beginning of the proof implies that $e_1$ is equivalent to $e_{j+1}$ in $P_j*Q$ and $P_{j+1}*Q$. Thus, $e_1(P_j*Q)e_1\cong e_{j+1}(P_j*Q)e_{j+1}$ and $e_1(P_{j+1}*Q)e_1\cong e_{j+1}(P_{j+1}*Q)e_{j+1}$.
Altogether,  $e_1(P_j*Q)e_1\cong Pe_{j+1}*e_1(P_{j+1}*Q)e_1$. This implies that \eqref{freefactor} holds for $j+1$  and, by induction, proves \eqref{freefactor}.
For $j=k$, \eqref{freefactor} gives the claim.
\end{proof}

To prove the proposition, assume first that $I$ is finite. Take $I=\{1,\cdots,n\}$, for some $n\geq 2$. For $1\leq i\leq n$, put $B_i=\mathbb Cp_{i,1}\oplus\cdots\oplus\mathbb Cp_{i,k}$ and  $C_i=B_1*\cdots *B_i*A_{i+1}*\cdots*A_n.$
We claim that
\begin{equation}\label{compress2}
\text{$p_{i,1}Mp_{i,1}\cong (*_{1\leq l\leq i,1\leq j\leq k}A_lp_{l,j})*p_{i,1}C_ip_{i,1}$, for every $1\leq i\leq n$.}
\end{equation}
The case $i=1$ follows from Claim \ref{compress}. Assume that \eqref{compress2} holds for some $1\leq i\leq n-1$. Since the projections $p_{i,1}$ and $p_{i+1,1}$ are equivalent in $C_i$ by the observation made in the beginning of the proof of Claim \ref{compress}, we get that $p_{i,1}Mp_{i,1}\cong p_{i+1,1}Mp_{i+1,1}$ and $p_{i,1}C_ip_{i,1}\cong p_{i+1,1}C_ip_{i+1,1}$. 
By applying Claim \ref{dykema} to $C_i=A_{i+1}*(B_1*\cdots*B_i*A_{i+2}*\cdots*A_k)$ and the projections $(p_{i+1,j})_{j=1}^k\subset A_{i+1}$, 
we get that $p_{i+1,1}C_ip_{i+1,1}\cong (*_{1\leq j\leq k}A_{i+1}p_{i+1,j})*p_{i+1,1}C_{i+1}p_{i+1,1}$. The last three isomorphisms together imply that \eqref{compress2} holds for $i+1$. By induction, this proves \eqref{compress2}.

Next, \eqref{compress2} for $i=n$ gives that $M^{1/k}\cong (*_{1\leq i\leq n,1\leq j\leq k}A_ip_{i,j})*p_{n,1}C_np_{n,1}$. We will prove that \begin{equation}\label{C_n}p_{n,1}C_np_{n,1}\cong L\mathbb F_{nk(k-1)-k^2+1}\end{equation} and thus finish the proof of case (1) by analyzing three separate cases.

If $n=k=2$, then $C_2\cong L\mathbb Z\otimes\mathbb M_2(\mathbb C)$ and  \cite[Proposition 3.2]{Dy94} impies that  $p_{2,1}C_2p_{2,1}\cong L\mathbb Z$. 
If $n>2$ or $k>2$, then
$C_n\cong L(*_{i=1}^n\mathbb Z/k\mathbb Z)$ is a II$_1$ factor. 
Since $\tau(p_{n,1})=1/k$, we get that $p_{n,1}C_np_{n,1}\cong L(*_{i=1}^n\mathbb Z/k\mathbb Z)^{1/k}$. 
Assume first that $k=2$ and $n>2$. Recall that $L(*_{j=1}^2\mathbb Z/2\mathbb Z)\cong L\mathbb Z\otimes\mathbb M_2(\mathbb C)$ and  $(A\otimes\mathbb M_2(\mathbb C))*L(\mathbb Z/2\mathbb Z)\cong (A*L\mathbb F_2)\otimes\mathbb M_2(\mathbb C)$, for every tracial von Neumann algebra $A$, by \cite[Theorem 3.5 (ii)]{Dy94}. Combining these facts with \eqref{interpolated} and using induction gives that $L(*_{i=1}^n\mathbb Z/2\mathbb Z)\cong L\mathbb F_{n/2}$, thus $L(*_{i=1}^n\mathbb Z/2\mathbb Z)^{1/2}\cong L\mathbb F_{2n-3}$. Finally, assume that $k>2$. Then \cite[Corollary 5.3]{Dy93} gives that $L(\mathbb Z/k\mathbb Z*\mathbb Z/k\mathbb Z)\cong L\mathbb F_{2(1-1/k)}$, while \cite[Proposition 2.4]{Dy93} gives that $L\mathbb F_r*L(\mathbb Z/k\mathbb Z)\cong L\mathbb F_{r+1-1/k}$, for every $r>1$.  By combining these facts, we get that $L(*_{i=1}^n\mathbb Z/k\mathbb Z)\cong L\mathbb F_{n(1-1/k)}$. Hence, using \eqref{interpolated} we derive that $L(*_{i=1}^n\mathbb Z/k\mathbb Z)^{1/k}\cong L\mathbb F_{1+k^2[n(1-1/k)-1]}=L\mathbb F_{1+nk(k-1)-k^2}.$
This altogether proves \eqref{C_n}.

To treat case (2), assume that $I$ is infinite. Take $I=\mathbb N$. For $i\geq 0$, let $M_i=A_{2i+1}*A_{2i+2}$. By applying case (1), we get that $M_i$ is a II$_1$ factor and $M_i^{1/k}\cong(*_{i\leq l\leq i+1,1\leq j\leq k}A_lp_{l,j})*L\mathbb F_{(k-1)^2}$, for every $i\geq 0$. Since $M=*_{i\geq 0}M_i$, \cite[Theorem 1.5]{DyRa} implies that $M^{1/k}\cong*_{k\geq 0}M_i^{1/k}$.  Thus, $M^{1/k}\cong (*_{1\leq i,1\leq j\leq k}A_ip_{i,j})*L\mathbb F_\infty$. Since $*_{1\leq i,1\leq j\leq k}A_ip_{i,j}$ is a free product of infinitely many II$_1$ factors, it freely absorbs $L\mathbb F_\infty$ by \cite[Theorem 1.5]{DyRa}. This finishes the proof of case (2).
\end{proof}

We say that an abelian tracial von Neumann algebra $(A,\tau)$ is {\it homogeneous} if for every $k\in\mathbb N$, there exists a partition of unity into $k$ projections $p_1,\cdots,p_k\in A$ such that for every $1\leq i\leq k$ we have that $\tau(p_i)=1/k$ and $(Ap_i,k\;\tau_{|Ap_i})$ is isomorphic to $(A,\tau)$. A homogeneous abelian von Neumann algebra is necessarily diffuse. Also, note that $L\mathbb Z$ and $(L\mathbb Z)^\omega$ are homogenenous, and that the direct sum of two homogeneous abelian von Neumann algebras is homogeneous.

\begin{cor}\label{amplify} Let $A$ be a homogeneous abelian tracial von Neumann algebra.
Then we have: 
\begin{enumerate}

 \item 
 If $2\leq n<\infty$ and $k\geq 1$, then $(A^{*n})^{1/k} \simeq A^{nk} * L\mathbb F_{1+nk(k-1)-k^2}$. 

 \item $\mathbb Q\subset\mathcal F(A^{*\infty})$.
 
 \end{enumerate}
\end{cor}

\begin{proof}
Part (1) follows from Proposition \ref{free_product_abelian}. Proposition \ref{free_product_abelian} also implies that $1/k\in\mathcal F(A^{*\infty})$, for every $k\in\mathbb N$, and thus part (2) also follows.
\end{proof}

When $A$ is separable (and thus $A\cong L\mathbb Z$), Corollary \ref{amplify} recovers two  results of Voiculescu \cite{Vo_amplification}: the amplification formula $L\mathbb F_n^{1/k}\cong L\mathbb F_{nk^2-k+1}$ and the fact that $\mathbb Q\subset\mathcal F(L\mathbb F_\infty)$.
Corollary \ref{amplify} extends these results to non-separable homogenenous abelian von Neumann algebras $A$. 
Recall that Radulescu \cite{Rad_F_infty} showed that in fact $\mathcal F(L\mathbb F_\infty)=\mathbb R_+^*$. By analogy with this result, we expect that $\mathcal F(A^{*\infty})=\mathbb R_+^*$, for any  homogenenous abelian von Neumann algebras $A$.

\bibliographystyle{alpha} 
\bibliography{bib}

\end{document}